\documentclass[11pt]{article}
\usepackage[T1]{fontenc}
\usepackage[utf8]{inputenc}
\usepackage{lmodern}
\usepackage{amsmath,amssymb,amsthm,mathtools}
\usepackage{enumitem}
\usepackage{microtype}
\usepackage[a4paper,margin=30mm]{geometry}
\usepackage[hidelinks]{hyperref}

\newtheorem{theorem}{Theorem}[section]
\newtheorem{lemma}[theorem]{Lemma}
\newtheorem{proposition}[theorem]{Proposition}
\newtheorem{corollary}[theorem]{Corollary}
\theoremstyle{definition}
\newtheorem{definition}[theorem]{Definition}
\newtheorem{example}[theorem]{Example}
\theoremstyle{remark}
\newtheorem{remark}[theorem]{Remark}

\newcommand{\R}{\mathbb{R}}
\newcommand{\Rp}{\mathbb{R}_{+}}
\newcommand{\ip}[2]{\langle #1,#2\rangle}
\newcommand{\norm}[1]{\lVert #1\rVert}

\newcommand{\angdist}{\operatorname{ad}}

\title{Characterizations of subdual cones via isotonicity of the norm of the metric projection and via antitonicity of angular distance}

\author{S. Z. N\'emeth\thanks{School of Mathematics, University of Birmingham, Watson Building, Edgbaston, Birmingham B15 2TT, United Kingdom. Email: \texttt{s.nemeth@bham.ac.uk}.}}

\begin{document}
\maketitle

\begin{abstract}
Let $H$ be a real Hilbert space, let $K\subseteq H$ be a closed convex cone, let $K^*$ be its dual closed convex cone, and let $P_K$ denote the
metric projection onto $K$.  We study scalar quantities associated with the projection $P_K$ and their interaction with the preorder induced by
$K$.  Our main result characterizes subduality by the monotonicity of the norm of a projection: $K\subseteq K^*$ if and only if $x\le_K y$ implies
$\norm{P_Kx}\le \norm{P_Ky}$.  We also discuss the sublinearity of the norm of a projection
and of the distance from a closed convex cone, relate these
functions to asymmetric seminorms, and introduce a normalized angular distance from a closed convex cone.  The angular distance admits equivalent
formulas in terms of $P_K$ and the ordinary distance function and, for closed convex cones with $K\ne H$, yields a second characterization of subduality.  Finally, coordinatewise monotone norms on finite-dimensional spaces are used to aggregate distances from several closed convex cones and thereby construct asymmetric seminorms.  The proofs are consequences of Moreau's decomposition theorem and the geometry of metric projections.
\end{abstract}

\noindent\textbf{Keywords:} closed onvex cone, pointed closed convex cone, subdual cone, metric projection, isotone mapping, isotone function, support
function, asymmetric norm, distance from a cone, angular distance from a cone.

\section{Introduction}

Metric projections onto convex pointed closed convex cones interact in a particularly rich way with the order structure induced by the pointed closed convex cone.  If $K$ is a pointed closed convex in a Hilbert space, then the relation
\[
 x\le_K y \quad\Longleftrightarrow\quad y-x\in K
\]
defines a partial order.  The requirement that the metric projection $P_K$ itself preserve this order,
\[
 x\le_K y \quad\Longrightarrow\quad P_Kx\le_K P_Ky,
\]
is strong and leads to the theory of isotone projection cones.  This theory was developed by Isac and A. B. N\'emeth in connection with complementarity problems; see, for example, \cite{IsacNemeth1986,IsacNemeth1990a,IsacNemeth1990b,IsacNemeth1990c,IsacNemeth1992}.  Related characterizations of latticial and subdual latticial pointed closed convex cones in Hilbert spaces were subsequently obtained in \cite{Nemeth2010a,Nemeth2010b}.

Full $K$-isotonicity is restrictive, so we ask what can already be inferred from the simpler function $x\mapsto\norm{P_Kx}$.  More precisely, we ask when
\[
 x\le_K y \quad\Longrightarrow\quad \norm{P_Kx}\le \norm{P_Ky}.
\]
One of the main observations of the paper is that this apparently much weaker requirement still has an exact geometric meaning: it is equivalent to the subduality condition $K\subseteq K^*$.  In other words, subduality is precisely the condition under which moving in a direction of $K$ cannot decrease the norm of the projection onto $K$.

There is also a convex-analytic interpretation of the function $x\mapsto\norm{P_Kx}$.  In finite-dimensional Euclidean space, Schneider records the identity
\[
 \norm{P_Kx}=h_{K\cap B}(x),
\]
where $h_{K\cap B}$ is the support function of the intersection of the pointed closed convex cone $K$ with the closed unit ball; see \cite[equation~(1.9)]{Schneider2022}.  The same identity holds in a Hilbert space by the same elementary projection argument.  Consequently, positive homogeneity and subadditivity of $\norm{P_K}$ are consistent with the general theory of support functions; see also \cite{Rockafellar1970,BauschkeCombettes2017}.  We nevertheless retain direct proofs based on Moreau's decomposition theorem, since they keep the order and projection geometry explicit.

The distance to $K$ is linked with the norm of the projection by Moreau's decomposition.  
Indeed,
\[
 d(x,K)=\norm{P_{K^\circ}x},
\]
where throughout this paper $K^\circ=-K^*$ denotes the polar closed convex cone.  Thus the norm of the projection and the distance are the norms of the two orthogonal components in the same decomposition, and it is useful to consider them together.  Distance functions associated with pointed closed convex cones are closely related, up to a sign convention for the ordering cone, to the canonical half-norms of ordered Banach spaces studied by Robinson and Yamamuro \cite{RobinsonYamamuro1983}; they also fit into the modern theory of asymmetric normed spaces \cite{Cobzas2013}.  We include several elementary null-set characterizations because they arise directly from these projection identities and will be used later.

A further aim is to distinguish distance from direction.  Since a cone is unchanged under multiplication by a positive scalar, the direction of a vector relative to the cone is often more informative than its ordinary distance alone.  We therefore introduce a normalized angular distance taking values in $[0,1]$.  Moreau's orthogonal decomposition expresses this quantity in terms of $\norm{P_Kx}$ and $d(x,K)$, and the characterization of subduality by the norm of the projection then has a corresponding formulation in terms of angular distance.

Finally, when several cones are considered at the same time, one may ask whether the distances from them can be combined into a single function while preserving the properties established for an individual cone.  We show that a coordinatewise monotone norm on the nonnegative orthant provides a simple way to do this.  If the cones have trivial intersection, the resulting function is an asymmetric seminorm, and a suitable symmetry of the family of cones yields an ordinary norm.

\section{Preliminaries}

Let $(H,\ip{\cdot}{\cdot})$ be a real Hilbert space and let $\norm{x}=\sqrt{\ip{x}{x}}$.  The Hilbert space setting is especially convenient here because metric projections and orthogonal decompositions are available.

We first recall the cone terminology used throughout the paper.  A nonempty closed set $K\subseteq H$ is called a \emph{closed convex cone} if
\[
 \lambda x+\mu y\in K
 \quad \mbox{whenever} \quad \mbox x,y\in K;\;\lambda,\mu\in\Rp:=\{x\in\mathbb R:x\ge 0\}.
\]
A closed convex cone $K$ induces a relation $\le_K$ on $H$ by
\[
 x\le_K y
 \quad\Longleftrightarrow\quad
 y-x\in K.
\]
Recall that a relation $\le$ is called a \emph{preorder} if it is reflexive and
transitive, and a \emph{partial order} if, in addition, it is antisymmetric.

The relation $\le_K$ is a preorder. Indeed, since $0\in K$, we have
$x\le_K x$ for every $x\in H$. Moreover, if $x\le_K y$ and $y\le_K z$, then
\[
 z-x=(z-y)+(y-x)\in K,
\]
and hence $x\le_K z$.

The preorder $\le_K$ is compatible with the vector space structure in the
following sense:
\[
 x\le_K y
 \quad\Longrightarrow\quad
 x+z\le_K y+z.
 \qquad \forall z\in H,
\]
and
\[
 x\le_K y
 \quad\Longrightarrow\quad
 \lambda x\le_K \lambda y,
 \qquad \forall\lambda\in\Rp.
\]

A closed convex cone $K$ is called \emph{pointed} if
\[
 K\cap(-K)=\{0\}.
\]
The preorder $\le_K$ is a partial order if and only if $K$ is pointed.
Indeed, if $x\le_K y$ and $y\le_K x$, then
\[
 y-x\in K\cap(-K),
\]
so pointedness implies $x=y$. Conversely, if $\le_K$ is antisymmetric and
$x\in K\cap(-K)$, then $0\le_K x$ and $x\le_K0$, and therefore $x=0$.

Both the dual and the polar cone will be used below. The dual cone is central to the notion of subduality, whereas the polar cone arises naturally in Moreau's decomposition. They are, respectively,
\[
 K^*=\{x\in H:\ip{x}{y}\ge0\ \text{for every }y\in K\},
 \qquad
 K^\circ=-K^*.
\]
The relation between $K$ and $K^*$ is central to the main result.  The closed convex cone $K$ is \emph{subdual} if $K\subseteq K^*$. It can be simply verified that a subdual closed convex cone is also pointed.

We shall use both the nearest point in the cone and the distance to it.  For $x\in H$, the distance from $x$ to $K$ is
\[
 d(x,K)=\min_{y\in K}\norm{x-y}.
\]
Since $K$ is closed and convex, there is a unique point $P_Kx\in K$ satisfying
\[
 \norm{x-P_Kx}=d(x,K).
\]

The basic link between projection onto a cone and projection onto its polar is provided by Moreau's decomposition theorem.  We shall use it repeatedly; see \cite{Moreau1962} and also \cite{BauschkeCombettes2017}.

\begin{theorem}[Moreau's decomposition theorem]
Let $K\subseteq H$ be a closed convex cone.  For $z,x,y\in H$, the following statement are equivalent:
\begin{enumerate}[label=\textup{(\roman*)}]
\item $z=x+y$, $x\in K$, $y\in K^circ$, and $\ip{x}{y}=0$;
\item $x=P_Kz$ and $y=P_{K^\circ}z$.
\end{enumerate}
In particular,
\[
 x=P_Kx+P_{K^\circ}x,
 \qquad
 \ip{P_Kx}{P_{K^\circ}x}=0.
\]
\end{theorem}

Denote $K^{\circ\circ}=(K^\circ)^\circ$. Note that if $K$ is a closed convex cone, then
\begin{equation}\label{bipolarity}
	K^{\circ\circ}=K.
\end{equation}
This identity is usually proved using a separation theorem, but it also follows by applying Moreau's decomposition theorem
twice: first to the cone $K$, and then to its polar cone $K^\circ$. Note that the proof of Moreau's decomposition does not 
need any separation theorem. We shall refer to the identity \eqref{bipolarity} as bipolarity equation. 

The monotonicity properties considered below are taken with respect to the preorder induced by the closed convex cone $K$.  A mapping $F:H\to H$
is called $K$-\emph{isotone} if $x\le_Ky$ implies $F(x)\le_KF(y)$.  A function $f:H\to\R$ is called $K$-\emph{isotone} if $x\le_Ky$ implies $f(x)\le f(y)$, and $K$-\emph{antitone} if $x\le_Ky$ implies $f(y)\le f(x)$.

Distance from a cone is generally not symmetric under $x\mapsto -x$.  This is why asymmetric seminorms provide an appropriate framework for some
of the functions considered below.  We shall use the terminology of the original formulation. Let any $\lambda\in\R_+$ and any $x,y\in H$.  A
function $q:H\to\R$ is called sublinear if it is positively homogeneous, that is,
\[
	q(\lambda x)=\lambda q(x)
\]
and subadditive, that is,
\[
	q(x+y)\le q(x)+q(y).
\]
A sublinear function $q:H\to\Rp$ is called an \emph{asymmetric seminorm} if
\[
 q(x)=q(-x)=0\quad\Longrightarrow\quad x=0.
\]
Terminology varies in the literature.  A functional satisfying precisely this definiteness
condition is commonly called an \emph{asymmetric norm} in modern treatments such as 
\cite{Cobzas2013}; closely related functionals are called half-norms in the ordered Banach-space literature \cite{RobinsonYamamuro1983}.  We retain the term ``asymmetric seminorm'' used
here. 

An asymmetric seminorm \(q:H\to\mathbb{R}_+\) is called a \emph{norm} if and only if it
is symmetric, that is, $q(x)=q(-x)$ for every $x\in H$.

\section{The norm of the metric projection}

We begin with a monotonicity property that holds for every closed convex cone.  The relevant order in this general case is the one induced by the dual cone.

\begin{lemma}\label{lem:dual-isotone}
Let $K\subseteq H$ be a closed convex cone.  Then the function
\[
 H\ni x\longmapsto \norm{P_Kx}
\]
is $K^*$-isotone.
\end{lemma}

\begin{proof}
Let $x\le_{K^*}y$.  The assertion $\|P_Kx\|\le\|P_Ky\|$ is immediate if $P_Kx=0$, so suppose that $P_Kx\ne0$.  By Moreau's decomposition theorem,
\[
 P_Ky-y=-P_{K^\circ}y\in K^*,
\]
and hence $y\le_{K^*}P_Ky$.  Therefore $x\le_{K^*}P_Ky$.  Using again
\[
 x=P_Kx+P_{K^\circ}x,
\]
we obtain
\[
 P_Ky-P_Kx-P_{K^\circ}x\in K^*.
\]
Taking the inner product with $P_Kx\in K$, and using Moreau orthogonality, yields
\[
 \norm{P_Kx}^2
 \le \ip{P_Kx}{P_Ky}
 \le \norm{P_Kx}\,\norm{P_Ky}.
\]
Since $P_Kx\ne0$, division by $\norm{P_Kx}$ gives
\[
 \norm{P_Kx}\le\norm{P_Ky}.
\]
\end{proof}

Lemma~\ref{lem:dual-isotone} shows that the norm of the projection is $K^*$-isotone for every closed convex cone $K$. The natural
question is therefore to determine for which cones it is also $K$-isotone. The following theorem gives an exact answer:
this happens precisely when $K$ is subdual.

\begin{theorem}\label{thm:subdual-projection-norm}
Let $K\subseteq H$ be a pointed closed convex cone.  Then $K$ is subdual if and only if the function $x\mapsto\norm{P_Kx}$ is $K$-isotone.
\end{theorem}

\begin{proof}
Suppose first that $K$ is subdual.  Then $K\subseteq K^*$, so
\[
 x\le_Ky\quad\Longrightarrow\quad x\le_{K^*}y.
\]
Lemma~\ref{lem:dual-isotone} therefore gives
\[
 \norm{P_Kx}\le\norm{P_Ky}.
\]

Conversely, suppose that $x\mapsto\norm{P_Kx}$ is $K$-isotone, and let $x\in K$.  Since
\[
 -x\le_K0,
\]
we have
\[
 0\le\norm{P_K(-x)}\le\norm{P_K0}=0.
\]
Thus $P_K(-x)=0$.  Moreau's decomposition theorem implies $-x\in K^\circ=-K^*$, and hence $x\in K^*$.  Therefore $K\subseteq K^*$.
\end{proof}

The norm of the projection also has a standard convex-analytic interpretation, which gives another viewpoint on the function appearing in Theorem~\ref{thm:subdual-projection-norm}.

\begin{remark}[Support-function interpretation]\label{rem:support}
Let $B=\{z\in H:\norm{z}\le1\}$.  For every closed convex cone $K$ and every $x\in H$,
\[
 \norm{P_Kx}
 =\sup_{z\in K\cap B}\ip{x}{z}.
\]
Indeed, Moreau's decomposition gives
\[
 \ip{x}{z}
 =\ip{P_Kx}{z}+\ip{P_{K^\circ}x}{z}
 \le\ip{P_Kx}{z}
 \le\norm{P_Kx}
\]
for $z\in K\cap B$, and equality is attained at $z=P_Kx/\norm{P_Kx}$ when $P_Kx\ne0$.  Thus $\norm{P_K\cdot}$ is the support function of $K\cap B$.  In finite dimensions this identity is recorded explicitly in \cite[Eq.~(1.9)]{Schneider2022}.  In particular, subadditivity of $\norm{P_K}$ is also an immediate consequence of the standard sublinearity of support functions \cite{Rockafellar1970}.
\end{remark}

The support-function representation already shows that $x\mapsto \|P_Kx\|$ is positively homogeneous and subadditive.  The
next result also gives a condition under which this function is an asymmetric seminorm.  We give a direct proof based on
Moreau's decomposition.

\begin{theorem}\label{thm:projection-sublinear}
Let $K\subseteq H$ be a closed convex cone and define $p(x)=\norm{P_Kx}$.  Then:
\begin{enumerate}[label=\textup{(\roman*)}]
\item $p$ is sublinear;
\item if $K^\circ$ is a pointed closed convex cone, then $p$ is an asymmetric seminorm.
\end{enumerate}
\end{theorem}

\begin{proof}
For $\lambda\ge0$, positive homogeneity of the projection onto a closed convex cone gives
\[
 P_K(\lambda x)=\lambda P_Kx,
\]
so $p(\lambda x)=\lambda p(x)$.

Let $x,y\in H$.  Moreau's decomposition theorem gives
\[
 x+y=P_Kx+P_Ky+P_{K^\circ}x+P_{K^\circ}y,
\]
and therefore
\[
 x+y\le_{K^*}P_Kx+P_Ky.
\]
Lemma~\ref{lem:dual-isotone} yields
\begin{align*}
 \norm{P_K(x+y)}
 &\le \norm{P_K(P_Kx+P_Ky)}\\
 &=\norm{P_Kx+P_Ky}\\
 &\le\norm{P_Kx}+\norm{P_Ky}.
\end{align*}
This proves subadditivity.

For (ii), suppose that $p(x)=p(-x)=0$.  Then $P_Kx=P_K(-x)=0$, and Moreau's decomposition theorem gives
\[
 x\in K^\circ,
 \qquad
 -x\in K^\circ.
\]
If $K^\circ$ is pointed, this implies $x=0$.
\end{proof}

\section{Distance functions and asymmetric seminorms}

The distance from a closed convex cone is the norm of the complementary Moreau decomposition component, so the preceding results for norms of projections apply immediately to distance functions:
\begin{equation}\label{eq:distance-polar-projection}
 d(x,K)=\norm{x-P_Kx}=\norm{P_{K^\circ}x}.
\end{equation}
This identity makes the next statement immediate from Theorem~\ref{thm:projection-sublinear}.

\begin{corollary}\label{cor:distance-sublinear}
	Let $K\subseteq H$ be a closed convex cone.  Then $d(\cdot,K)$ is sublinear.
\end{corollary}

In the preceding corollary $K=\{x\in H:d(x,K)=0\}$ because $K$ is closed. Therefore, it is natural to put the converse 
question: Which sets can occur as the zero set of a continuous, positively homogeneous, subadditive function? 

\begin{theorem}\label{thm:closed convex cone-null}
Let $K\subseteq H$.  Then $K$ is a closed convex cone if and only if there exists a continuous sublinear function 
$q:H\to\Rp$ such that
\[
 K=\{x\in H:q(x)=0\}.
\]
\end{theorem}

\begin{proof}
If $K$ is a closed convex cone, take $q=d(\cdot,K)$.  The function $q$ is continuous, its zero set is $K$, and
Corollary~\ref{cor:distance-sublinear} yields sublinearity. 

Conversely, suppose that $K$ is the zero set of a continuous sublinear function $q:H\to\Rp$.  Continuity of $q$ implies 
that $K$ is closed.  If $x,y\in K$ and $\lambda,\mu\ge0$, then
\[
 0\le q(\lambda x+\mu y)
 \le \lambda q(x)+\mu q(y)=0.
\]
Thus $\lambda x+\mu y\in K$, so $K$ is a closed convex cone.
\end{proof}

For a pointed cone, the corresponding additional condition is asymmetric definiteness: simultaneous vanishing at $x$ and $-x$ can occur only at the origin.

\begin{theorem}\label{thm:pointed closed convex cone-null}
Let $K\subseteq H$.  Then $K$ is a pointed closed convex cone if and only if there exists a continuous asymmetric seminorm $q:H\to\Rp$ such that
\[
 K=\{x\in H:q(x)=0\}.
\]
\end{theorem}

\begin{proof}
Suppose first that $K$ is a pointed closed convex cone and set $q=d(\cdot,K)$.  By Theorem~\ref{thm:closed convex cone-null}, $q$ is continuousand sublinear.  If $q(x)=q(-x)=0$, then $x,-x\in K$, hence
\[
 x\in K\cap(-K)=\{0\}.
\]
Thus $q$ is an asymmetric seminorm.

Conversely, suppose that $K$ is the zero set of a continuous asymmetric seminorm $q$.  Theorem~\ref{thm:closed convex cone-null} shows that $K$ is a closed convex cone.  If $x\in K\cap(-K)$, then
\[
 q(x)=q(-x)=0,
\]
so $x=0$.  Hence $K$ is pointed.
\end{proof}

The same zero-set argument is not specific to the metric projection.  It also applies to any continuous retraction whose displacement has the same
sublinearity property.

\begin{corollary}\label{cor:retraction}
Let $K\subseteq H$ and let $\rho:H\to K$ be a continuous retraction.  If the function
\[
 x\longmapsto\norm{x-\rho(x)}
\]
is sublinear, then $K$ is a closed convex cone.
\end{corollary}

\begin{proof}
The zero set of $x\mapsto\norm{x-\rho(x)}$ is exactly $K$.  The conclusion follows from Theorem~\ref{thm:closed convex cone-null}.
\end{proof}

If the function in Theorem~\ref{thm:closed convex cone-null} is taken to be the distance itself, one obtains an intrinsic characterization among closed sets.

\begin{theorem}\label{thm:distance-closed convex cone}
Let $K\subseteq H$ be a nonempty closed set.  Then $K$ is a closed convex cone if and only if $d(\cdot,K)$ is sublinear.
\end{theorem}

\begin{proof}
The forward implication is Corollary~\ref{cor:distance-sublinear}.  Conversely, the distance function from any nonempty set is $1$-Lipschitz and hence continuous.  Since $K$ is closed,
\[
 K=\{x\in H:d(x,K)=0\}.
\]
Theorem~\ref{thm:closed convex cone-null} now applies.
\end{proof}

Adding the asymmetric definiteness condition to the distance detects pointedness as well.

\begin{theorem}\label{thm:distance-pointed closed convex cone}
Let $K\subseteq H$ be a nonempty closed set.  Then $K$ is a pointed closed convex cone if and only if $d(\cdot,K)$ is an asymmetric seminorm.
\end{theorem}

\begin{proof}
If $K$ is a pointed closed convex cone, Theorem~\ref{thm:distance-closed convex cone} implies sublinearity.  If
\[
 d(x,K)=d(-x,K)=0,
\]
then $x,-x\in K$, and pointedness implies $x=0$.

Conversely, if $d(\cdot,K)$ is an asymmetric seminorm, Theorem~\ref{thm:distance-closed convex cone} shows that $K$ is a closed convex cone.  If $x\in K\cap(-K)$, then
\[
 d(x,K)=d(-x,K)=0,
\]
so $x=0$ by asymmetric definiteness.
\end{proof}

\section{Angular distance from a closed convex cone}

The ordinary distance $d(x,K)$ changes when $x$ is multiplied by a positive scalar, whereas its direction relative to a cone does not.  For questions concerned with direction, one can instead measure the smallest angle between the ray generated by $x$ and the cone.

\begin{definition}\label{def:angular-distance}
Let $K\subseteq H$ be a closed convex cone and let $x\in H\setminus\{0\}$.  The \emph{normalized angular distance} from $x$ to $K$ is
\[
\angdist(x,K)=
\begin{cases}
\displaystyle
\frac{2}{\pi}
\inf_{y\in K\setminus\{0\}}
\arccos\!\left(
\frac{\ip{x}{y}}{\norm{x}\norm{y}}
\right),
& x\notin K^\circ,\\[3ex]
1,&x\in K^\circ.
\end{cases}
\]
\end{definition}

Thus acute angles are normalized by the factor $2/\pi$, while directions making a nonacute angle with every vector of $K$ are assigned the value $1$.  Although the definition involves an infimum over the cone, the projection identifies the relevant direction and gives several equivalent formulas.

\begin{proposition}\label{prop:angular-formulas}
Let $K\subseteq H$ be a closed convex cone and $x\in H\setminus\{0\}$.  Then:
\begin{enumerate}[label=\textup{(\roman*)}]
\item
\[
\angdist(x,K)=
\begin{cases}
\displaystyle
\frac{2}{\pi}\arccos\!\left(
\frac{\ip{x}{P_Kx}}{\norm{x}\norm{P_Kx}}
\right),&x\notin K^\circ,\\[3ex]
1,&x\in K^\circ;
\end{cases}
\]
\item
\begin{equation}\label{eq:angle-arccos}
\angdist(x,K)=
\frac{2}{\pi}\arccos\!\left(\frac{\norm{P_Kx}}{\norm{x}}\right)
=
\frac{2}{\pi}\arccos\!\left(\frac{d(x,K^\circ)}{\norm{x}}\right);
\end{equation}
\item
\begin{equation}\label{eq:angle-arcsin}
\angdist(x,K)=
\frac{2}{\pi}\arcsin\!\left(\frac{d(x,K)}{\norm{x}}\right);
\end{equation}
\item if $x\notin K^\circ$, then
\begin{equation}\label{eq:angle-arctan}
\angdist(x,K)=
\frac{2}{\pi}\arctan\!\left(\frac{d(x,K)}{\norm{P_Kx}}\right)
=
\frac{2}{\pi}\arctan\!\left(\frac{d(x,K)}{d(x,K^\circ)}\right).
\end{equation}
With the convention $\arctan(+\infty)=\pi/2$, the same formula also covers $x\in K^\circ$;
\item $\angdist(x,K)=0$ if and only if $x\in K$, and $\angdist(x,K)=1$ if and only if $x\in K^\circ$;
\item
\[
 \angdist(x,K)+\angdist(x,K^\circ)=1.
\]
\end{enumerate}
\end{proposition}

\begin{proof}
Assume first that $x\notin K^\circ$.  Then $P_Kx\ne0$.  Since $\arccos$ is decreasing, Definition~\ref{def:angular-distance} reduces the problem to maximizing
\[
 \frac{\ip{x}{y}}{\norm{x}\norm{y}}
 \qquad(y\in K\setminus\{0\}).
\]
For such $y$, Moreau's decomposition theorem and the definition of the polar closed convex cone give
\begin{align*}
\frac{\ip{x}{y}}{\norm{x}\norm{y}}
&=
\frac{\ip{P_Kx}{y}+\ip{P_{K^\circ}x}{y}}{\norm{x}\norm{y}}\\
&\le
\frac{\ip{P_Kx}{y}}{\norm{x}\norm{y}}\\
&\le
\frac{\norm{P_Kx}}{\norm{x}}.
\end{align*}
Equality is attained by every $y=\lambda P_Kx$ with $\lambda>0$.  This proves (i).  Moreover, Moreau orthogonality can be 
reformulated as,
\[
 \ip{x}{P_Kx}=\norm{P_Kx}^2,
\]
which implies the first equality in \eqref{eq:angle-arccos}.  The second one follows from
\[
 d(x,K^\circ)=\norm{P_Kx}.
\]

Moreau orthogonality also gives the Pythagorean identity
\begin{equation}\label{eq:pythagoras}
 \norm{x}^2
 =\norm{P_Kx}^2+\norm{P_{K^\circ}x}^2
 =\norm{P_Kx}^2+d(x,K)^2.
\end{equation}
Hence, equations \eqref{eq:angle-arcsin} and \eqref{eq:angle-arctan} follow from \eqref{eq:angle-arccos} and elementary trigonometry.  The same identities also show that $\angdist(x,K)\in[0,1]$.

For proving (v), bear in mind that equation \eqref{eq:pythagoras} shows that $\norm{P_Kx}=\norm{x}$ if and only if $x\in K$, and use
\eqref{eq:angle-arccos}, while \eqref{eq:angle-arccos} also shows that $\angdist(x,K)=1$ precisely when $\norm{P_Kx}=0$, 
equivalently $x\in K^\circ$.

If $x\in K\cup K^\circ$, (vi) follows from (v).  Otherwise, apply \eqref{eq:angle-arctan} to $K$ and $K^\circ$ and 
use
\[
 \arctan a+\arctan(1/a)=\frac{\pi}{2}
 \qquad(a>0).
\]
Finally. if $x\in K^\circ$, then items (i)-(iv) are either trivial or follow from (v).
\end{proof}

\begin{remark}
Geometrically, if $x\notin K^\circ$, then $\angdist(x,K)$ is the smallest angle between $x$ and a nonzero vector of $K$, 
normalized by the factor $2/\pi$. If $x\in K^\circ$, so that the angle between $x$ and every nonzero vector of $K$ is at
least $\pi/2$, we set $\angdist(x,K)=1$.
\end{remark}

Angular quantities associated with convex cones have previously been studied by Iusem and 
Seeger, in particular maximal and critical angles determined by pairs of vectors in a cone 
\cite{IusemSeeger2005,IusemSeeger2008,IusemSeeger2009}.
Our angular distance is of a different type: it measures the angular separation of an
arbitrary vector from the cone.

Next, we note a simple property of the ordinary distance: moving from $x$ in a direction belonging to $K$ cannot increase 
the distance to $K$. No subduality assumption is needed for this property to hold.

\begin{lemma}\label{lem:distance-antitone}
Let $K\subseteq H$ be a closed convex cone.  Then $d(\cdot,K)$ is $K$-antitone.
\end{lemma}

\begin{proof}
If $x\le_Ky$, then $y-x\in K$.  By Corollary~\ref{cor:distance-sublinear},
\[
 d(y,K)
 =d((y-x)+x,K)
 \le d(y-x,K)+d(x,K)
 =d(x,K).
\]
\end{proof}

By the preceding lemma, Theorem~\ref{thm:subdual-projection-norm} now provides a corresponding characterization in terms of
the angular distance.  

\begin{theorem}\label{thm:angular-subdual}
Let $K\subseteq H$ be a pointed closed convex cone.  Then $K$ is subdual if and only if $\angdist(\cdot,K)$ is $K$-antitone on $H\setminus\{0\}$.
\end{theorem}

\begin{proof}
Suppose first that $K$ is subdual and let $x,y\ne0$ satisfy $x\le_Ky$.  By Theorem~\ref{thm:subdual-projection-norm},
\[
 \norm{P_Kx}\le\norm{P_Ky},
\]
while Lemma~\ref{lem:distance-antitone} gives
\[
 d(y,K)\le d(x,K).
\]
If neither $x$ nor $y$ lies in $K^\circ$, equation \eqref{eq:angle-arctan} and monotonicity of $\arctan$ therefore imply
\[
 \angdist(y,K)\le\angdist(x,K).
\]
The cases in which one of the points lies in $K^\circ$ follow from Proposition~\ref{prop:angular-formulas}(v).

Conversely, suppose that $\angdist(\cdot,K)$ is $K$-antitone and assume, for contradiction, that $K$ is not subdual.  Then there exist $u,v\in K$ such that
\[
 \ip{u}{v}<0.
\]
Since $K$ is a closed convex cone and $K\ne H$ (because it is pointed), the bipolarity equation \eqref{bipolarity} implies 
that its polar $K^\circ$ contains a nonzero vector; choose $0\ne y\in K^\circ$.  For sufficiently large $t>0$, the vector
\[
 x=y-tu
\]
does not belong to $K^\circ$, because
\[
 \ip{x}{v}=\ip{y}{v}-t\ip{u}{v}>0.
\]
Moreover, $y-x=tu\in K$, so $x\le_Ky$.  Proposition~\ref{prop:angular-formulas}(v) gives
\[
 \angdist(y,K)=1
 \qquad\text{and}\qquad
 \angdist(x,K)<1,
\]
contradicting $K$-antitonicity.  Hence $K$ is subdual.
\end{proof}

\section{Aggregating distances from several closed convex cones}

The previous sections associate a nonnegative sublinear distance function with each closed convex cone.  If several cones are present, we ask when their distances can be combined without losing subadditivity.  Since the relevant inequalities hold coordinate by coordinate, the norm used to combine them should respect coordinatewise order.

\begin{definition}
A norm $\zeta$ on $\R^n$ is called \emph{coordinatewise monotone on $\Rp^n$} if
\[
 0\le a_i\le b_i\quad(i=1,\ldots,n)
 \quad\Longrightarrow\quad
 \zeta(a_1,\ldots,a_n)\le\zeta(b_1,\ldots,b_n).
\]
\end{definition}

For $1\le p<\infty$, the $\ell_p$-norm on $\mathbb{R}^n$ is defined by
\[
\|x\|_p
=
\left(\sum_{i=1}^n |x_i|^p\right)^{1/p},
\]
while
\[
\|x\|_\infty
=
\max_{1\le i\le n}|x_i|.
\]

More generally, let $w_1,\ldots,w_n>0$. Then, the weighted $\ell_p$-norm is defined by
\[
\|x\|_{p,w}
=
\left(\sum_{i=1}^n w_i|x_i|^p\right)^{1/p},
\qquad 1\le p<\infty,
\]
while 
\[
\|x\|_\infty
=
\max_{1\le i\le n}w_i|x_i|.
\]

All these norms are coordinatewise monotone on $\mathbb{R}_+^n$.

The coordinatewise monotonicity condition is exactly what is needed to pass from the componentwise inequalities for the 
distances to an inequality for their combination.

\begin{theorem}\label{thm:aggregate}
Let $K_1,\ldots,K_n$ be closed convex cones in $H$ such that
\[
 K_1\cap\cdots\cap K_n=\{0\},
\]
and let $\zeta$ be a coordinatewise monotone norm on $\R^n$.  Define
\begin{equation}\label{eq:eta}
 \eta(x)
 =\zeta\bigl(d(x,K_1),\ldots,d(x,K_n)\bigr)
 =\zeta\bigl(\norm{P_{K_1^\circ}x},\ldots,
              \norm{P_{K_n^\circ}x}\bigr).
\end{equation}
Then $\eta$ is sublinear, and
\[
 \eta(x)=0\quad\Longleftrightarrow\quad x=0.
\]
In particular, $\eta$ is an asymmetric seminorm.
\end{theorem}

\begin{proof}
Positive homogeneity follows from positive homogeneity of each distance function and of the norm $\zeta$.  By Corollary~\ref{cor:distance-sublinear},
\[
 d(x+y,K_i)\le d(x,K_i)+d(y,K_i)
 \qquad(i=1,\ldots,n).
\]
From the coordinatewise monotonicity of $\zeta$, followed by its triangle inequality, we obtain
\begin{align*}
 \eta(x+y)
 &\le
 \zeta\bigl(d(x,K_1)+d(y,K_1),\ldots,
             d(x,K_n)+d(y,K_n)\bigr)\\
 &\le \eta(x)+\eta(y).
\end{align*}
If $\eta(x)=0$, then every coordinate in \eqref{eq:eta} is zero, so
\[
 x\in K_1\cap\cdots\cap K_n=\{0\}.
\]
The converse is immediate.
\end{proof}

\begin{remark}\label{rem:arbitrary-zeta}
The coordinatewise monotonicity of $\zeta$ cannot simply be omitted.  From the inequalities
\[
 d(x+y,K_i)\le d(x,K_i)+d(y,K_i)
\]
one only obtains a componentwise comparison in $\Rp^n$; an arbitrary norm on $\R^n$ need not preserve that comparison.  Likewise, symmetry of $\zeta$ as a norm on $\R^n$ does not by itself imply $\eta(-x)=\eta(x)$, because the individual distance functions are generally asymmetric with respect to $x\mapsto -x$.
\end{remark}

The construction is generally asymmetric.  If the family of cones is symmetric under multiplication by $-1$, however, the resulting function becomes symmetric as well.

\begin{corollary}\label{cor:aggregate-symmetric}
Under the assumptions of Theorem~\ref{thm:aggregate}, suppose in addition that there is a permutation $\pi$ of $\{1,\ldots,n\}$ such that
\[
 -K_i=K_{\pi(i)}\qquad(i=1,\ldots,n),
\]
and that $\zeta$ is invariant under the corresponding permutation of coordinates.  Then $\eta$ is a norm on $H$.
\end{corollary}

\begin{proof}
For every $i$,
\[
 d(-x,K_i)=d(x,-K_i)=d(x,K_{\pi(i)}).
\]
Thus the coordinate vector defining $\eta(-x)$ is a permutation of the coordinate vector defining $\eta(x)$.  Permutation invariance of $\zeta$ gives $\eta(-x)=\eta(x)$, and Theorem~\ref{thm:aggregate} supplies the remaining norm axioms.
\end{proof}

Note that any $\ell_p$-norm with $1\le p\le\infty$ is invariant under any permutation of coordinates. The simplest instance 
of cones in Corollary \ref{cor:aggregate-symmetric} is obtained by pairing a cone with its negative.

\begin{example}
Let $H=\mathbb{R}^2$ and consider the four closed convex cones

$$
K_1=\{(t,0):t\ge0\},\qquad
K_2=\{(-t,0):t\ge0\},
$$

$$
K_3=\{(0,t):t\ge0\},\qquad
K_4=\{(0,-t):t\ge0\}.
$$

Then

$$
K_1\cap K_2\cap K_3\cap K_4=\{0\},
$$

and

$$
-K_1=K_2,\qquad -K_2=K_1,\qquad
-K_3=K_4,\qquad -K_4=K_3.
$$

Thus the family of cones is invariant under multiplication by $-1$ up to the permutation

$$
\pi=(1\ 2)(3\ 4).
$$

If $\zeta$ is a coordinatewise monotone norm on $\mathbb{R}^4$ which is invariant under this permutation, then Corollary~\ref{cor:aggregate-symmetric} implies that

$$
\eta(x)
=
\zeta\bigl(
d(x,K_1),d(x,K_2),d(x,K_3),d(x,K_4)
\bigr)
$$

is a norm on $\mathbb{R}^2$.

For example, take $\zeta=\|\cdot\|_2$ the Euclidean norm. If $x=(x_1,x_2)$, then

$$
d(x,K_1)^2+d(x,K_2)^2=x_1^2+2x_2^2,
$$

and

$$
d(x,K_3)^2+d(x,K_4)^2=2x_1^2+x_2^2.
$$

Hence

$$
\eta(x)^2
=
d(x,K_1)^2+d(x,K_2)^2+d(x,K_3)^2+d(x,K_4)^2
=
3(x_1^2+x_2^2),
$$

so

$$
\eta(x)=\sqrt{3}\,\|x\|_2.
$$

\end{example}

\begin{example}
Let $K$ be a pointed closed convex cone and let $\zeta$ be a coordinatewise monotone norm on $\R^2$ satisfying $
\zeta(a,b)=\zeta(b,a)$ (for example any $\ell_p$-norm on $\R^2$). Then,
\[
 \eta(x)=\zeta\bigl(d(x,K),d(x,-K)\bigr)
\]
is a norm.  Indeed, $K\cap(-K)=\{0\}$ and the two closed convex cones are interchanged by multiplication by $-1$.
\end{example}

A different example is obtained from a cone and its polar.  Here no symmetry assumption is needed, because Moreau orthogonality gives the Hilbert norm directly.

\begin{example}
For every closed convex cone $K$, the Euclidean aggregation of the distances to $K$ and $K^\circ$ gives back the Hilbert norm:
\[
 \sqrt{d(x,K)^2+d(x,K^\circ)^2}
 =\sqrt{\norm{P_{K^\circ}x}^2+\norm{P_Kx}^2}
 =\norm{x},
\]
by Moreau orthogonality.
\end{example}

\section{Concluding remarks}

In this note, we studied scalar functions associated with the metric projection onto a closed convex cone and their monotonicity with respect to
the order induced by the cone. For a pointed closed convex cone $K$, we proved that $K$ is 
subdual if and only if the norm of the projection $\norm{P_K}$ is $K$-isotone. This contrasts
with the $K$-isotonicity of the metric projection $P_K$ itself, which is a substantially
stronger property and is closely related to lattice-type properties of the cone.

Moreau's decomposition relates the distance from a cone to the norm of the projection onto 
its polar cone. Using this relation, we proved that a nonempty closed set $K$ is a closed 
convex cone if and only if $d(\cdot,K)$ is positively homogeneous and subadditive, and that
$K$ is a pointed closed convex cone if and only if $d(\cdot,K)$ is an asymmetric seminorm.
We also introduced a normalized angular distance and proved that, for a pointed closed 
convex cone, subduality is equivalent to the $K$-antitonicity of this angular distance on 
$H\setminus\{0\}$.

The identity expressing the norm of a projection as the support function of $K\cap B$ 
relates the
first characterization of subduality to support functions, while the results for the distance
function connect naturally with half-norms and asymmetric normed spaces. If $K_1,\ldots,K_n$ 
are closed convex cones with trivial intersection and $\zeta$ is a coordinatewise monotone
norm, then applying $\zeta$ to the distances from these cones yields an asymmetric seminorm.
If in addition reflection with respect to the origin permutes the cones and $\zeta$ is
invariant under the corresponding permutation of coordinates, the resulting function is a 
norm.

Several questions remain open. One may ask which other scalar functions associated with 
$P_Kx$ have monotonicity properties that are equivalent to geometric properties of $K$, and
which nonlinear functions of several cone distances preserve positive homogeneity and
subadditivity.

The $K$-antitonicity of the normalized angular distance for subdual cones, together with its
invariance under positive scaling, also raises the question of whether an appropriate order 
can be introduced on the unit sphere for which a related monotonicity property holds. A 
difficulty is that $x\le_K y$ does not in general imply a corresponding order relation
between $x/\norm{x}$ and $y/\norm{y}$.

\section*{Declaration of use of AI} 

ChatGPT Plus was used to assist with exposition, including plain-language explanations and simple 
illustrative examples, and to identify minor mathematical corrections. The mathematical 
ideas, argument, proof, and conclusions are my own; I verified and incorporated the suggested 
corrections.

\end{document}